\documentclass[11pt]{article}
\usepackage{fullpage}
\usepackage{amsmath, amssymb}
\usepackage{mathtools}
\usepackage{multicol}
\usepackage[utf8]{inputenc}
\usepackage[english]{babel}
\usepackage{indentfirst}
\usepackage[margin=1in,includefoot]{geometry}
\usepackage{graphicx}
\graphicspath{ {/Users/nsapkota/Desktop/SumRec/fig/} }
\usepackage{float}
\usepackage{parskip}
\usepackage{ragged2e}
\usepackage{amsthm}
\usepackage{subfig}
\usepackage{cite}

\theoremstyle{definition}
\newtheorem{defn}{Definition}[section]

\usepackage{color}

\usepackage[pdftitle={Predator Prey Interaction Model with Hunting Cooperation among Predators and Allee Effect in Prey}, pdfauthor={Aaditya Kharel},
colorlinks=true, pdfstartview=FitV, linkcolor=blue, citecolor=blue, urlcolor=blue]{hyperref}

\newtheorem{theorem}{Theorem}[section]

\title{\textbf{Predator Prey Interaction Model with Hunting Cooperation among Predators and Allee Effect in Prey}}
\author{\text{Aaditya Kharel*,  Michelle McCullum, Nick Burks}\\ 
Advisors: Zhifu Xie*, Huiqing Zhu*\\
*Email: \href{}{aaditya.kharel@usm.edu}, zhifu.xie@usm.edu, huiqing.zhu@usm.edu\\
School of Mathematics and Natural Sciences \\
The University of Southern Mississippi\\
Hattiesburg, Mississippi 39406\\
}
\date{}

\begin{document}
\setlength{\parindent}{1cm}


%
%

\maketitle

%
%

\begin{abstract}
This paper investigates a dynamical predator-prey interaction model that incorporates: (a) hunting cooperation ($\alpha$) among predators; (b) Allee effect in prey. We show all possible boundary and interior solutions. In order to analyze the stability of the solution, we make use of the Jacobian matrix and the resultant characteristic polynomial. Particularly, the sign of the eigenvalue is used to determine stability of a solution. We then provide a proof for stability of interior solution. Finally, we verify our results numerically in MALTAB by plotting: (1) predator-prey intersection graphs; (2) prey-predator vs hunting cooperation graphs; (3) initial condition trajectory for equilibrium solution. It is interesting to notice that the hunting cooperation can switch the stability of coexistence equilibrium solutions. Through numerical simulations, it was verified that increasing the hunting cooperation could lead to extinction of both prey and predator population for $\alpha$$>$$0.96$, given our choice of parameters.  \\


\end{abstract}

%
%

\section{Introduction and Model Description}

It is important that the predators and prey in any ecosystem coexist such that the presence of either of the species is not detrimental to their coexistence. Rather, the presence of both predator and prey species should provide a check and balance mechanism for each species to coexist and maintain their population in the ecosystem towards stable equilibrium over time. Some examples of such predator-prey interactions that help maintain harmony in the ecosystem are wolf-rabbit, tiger-deer, lion-zebra and so on. The extinction of a species, whether predator or prey, is undesirable to any ecosystem and thus can create imbalance in population dynamics in the long run.

Many factors such as Allee effects, dispersal movements of both prey and predator (diffusive model), hunting cooperation among predator species, function response between prey and predator (consumption of prey per predator) and among many others, play a significant role in determining the population dynamics of predator and prey to create stable coexistence equilibrium. In particular, hunting cooperation and Allee effect phenomena have separately been investigated under various constraints. Allee effect in prey is said to occur if the population of prey crosses a given population threshold (known as Allee threshold), which then leads to a stable equilibrium coexistence. If the prey population is below Allee threshold, it leads to extinction of the prey population.

In Rao and Kang's paper, they propose a reaction-diffusion predator prey model with Allee effects in prey and diffusion in both prey and predator to investigate how Allee effect and diffusion of species affect the spatial-temporal dynamics (Rao and Kang 2015 \cite{RK}).  On the other hand, Alves and Hilker discovered that hunting cooperation among predators can induce Allee effects (Alves and Hilker 2017\cite{AH}). Then the aim of our research is to study the effects of incorporating hunting cooperation term into the non-diffusive version of Rao and Kang's model to see if there is still an Allee effect phenomena when hunting cooperation term is added to their model.

As a result, we propose a dynamical system of predator-prey interaction model consisting of two ordinary differential equations (ODEs) that incorporates hunting cooperation among the predators into Rao and Kang's model. Our proposed model is:

\begin{equation}\label{eq1}
\begin{split}
\frac{du}{dt} &=au\big(u-b\big)\big(1-u\big)-\frac{uv\big(1+\alpha v\big)}{v+\big(1+\alpha v\big)u}, \\
\frac{dv}{dt} &=\frac{cuv\big(1+\alpha v \big)}{v+\big(1+\alpha v \big)u}-dv, 
\end{split}
\end{equation}
where $u$ is the Prey Population,
$v$ is the Predator Population,
$a$ is the intrinsic growth rate of the Predator,
$b$ is the Allee threshold and $b \in (-1,1),$
$c$ is the energy conversion coefficient,
$d$ is the death rate of the Prey and
$\alpha$ is the hunting cooperation.

We particularly want to investigate if hunting cooperation causes any Allee effect phenomena in prey or predator. This is particularly important because if Allee effect is present in a two-species interaction model, we can predict that prey population below Allee threshold will not survive at all in any ecosystem due to Allee effect.

\section {Equilibrium Solution}
\begin{defn}
An equilibrium solution in the context of prey-predator interaction can be defined as an ordered pair of prey and predator population, which mathematically is a  solution to the system of ordinary differential equation, such that there is no rate of change of prey and predator population. \end{defn}
\begin{defn}
A boundary solution is a solution that lies either on the positive x-axis or the positive y-axis. 
\end{defn} 
\begin{defn}
An interior solution is a solution that lies in the first quadrant excluding x-axis and y-axis.
\end{defn} 
\begin{theorem}
Assume all parameters in the model \eqref{eq1} are positive. The system always has three boundary equilibrium solutions: $(0,0), (1,0)$ and $(b,0)$. The system may have none, one  or two interior equilibrium solutions which  depends on the parameters. 
\end{theorem}
{\it Proof:} We divide our equilibrium solution into boundary and interior solution. We ignore all other quadrants because the population of the prey and predator cannot be negative. In order to find a boundary solution, we first assume that $u=0$ and $v \neq 0$. In the second case, we then assume that $v=0$ and $u \neq 0$. Finally, we assume that both $u=0$ and $v=0$ and try to plug it into our model. However, in the final case, the limit as $u$ and $v$ approaches zero does not exist. Nevertheless, from our numerical simulation, it is evident that $(0,0)$ is our equilibrium solution. So the boundary solution for our system is $(0,0), (1,0)$ and $(b,0)$. 

Likewise, the interior solution is the solution of the following system of nonlinear equation which is obtained by multiplying first equation from the model in equation (1) by c and adding two equations of the system to obtain one of the equations of v in terms of u. Similarly, we solve of v in terms of u purely from the second equation from the model in equation (1). The non-linear equation that represents $v$ in terms of $u$ is given below:
\begin{equation}\label{eq2}
v=\frac{u\Big(c-d\Big)}{d-\alpha u\Big(c-d\Big)}.
\end{equation}
$u$ is a positive root of a third order polynomial  $$\Big( \frac {ac}{d} u - b \Big) \Big(1-u \Big) \Big( d-\alpha u (c-d)\Big)=\Big(c-d\Big).$$
This non-linear system helps us solve for an interior solution (u,v) given values for our constant $a, b, c, d$.  In our research, we used the $fsolve$ function in MATLAB to solve for the system when the parameters $a, b, c, d$ are given. Our numerical results show that there are none, one or two interior solutions. 

\section{Stability Analysis}
In this section, we derive two variants of the Jacobian matrix $(J)$ for the model that is tuned to effectively aid in the stability analysis of each boundary and interior solutions. We then present both sufficient and necessary conditions for stability of equilibrium solution for our system. In general, the characteristic polynomial for a 2x2 system can be given as $\lambda ^2 -trace (J) \lambda + det (J)$, where, 
\begin{equation}
\begin{split}
\lambda=\frac{trace \pm \sqrt{trace^2- 4det(J)}}{2}
\end{split}
\end{equation}
is our eigenvalue. To guarantee stability of a solution, we must ensure that our eigenvalue is negative for real eigenvalues or ensure that the real part of the complex eigenvalue is negative for complex eigenvalues. This leads us to the following necessary and sufficient conditions to guarantee stability. 
\begin{itemize}
  \item $Trace (J)=J_{11}+J_{22}$ must be less than zero
  \item $Determinant (J)=J_{11}J_{22}-J_{12}J_{21}$ must be greater than zero
\end{itemize}

\subsection{Stability for Boundary Solutions}
\begin{theorem} The boundary solutions $(1,0)$ is always stable if $c<d$ and $(b,0)$ is always unstable if $b>0$ and $c<d$.
\end {theorem}
\begin{proof}
We first compute the general form of Jacobian matrix for an equlibrium solution as follows:
\[J=
\begin{pmatrix}
a(-3u^2+2u+2bu-b)-\frac{v^2(1+ \alpha v)}{[u(1+ \alpha v)+v]^2} &   \frac{u(uv^2 \alpha^2+\alpha v^2+2 \alpha uv+u}{[u(1+\alpha v)+v]^2}
\\ \\
\frac{cv^2(1+ \alpha v)}{[u(1+ \alpha v)+v]^2} &   \frac{cu(uv^2 \alpha^2 + \alpha v^2+2 \alpha uv +u)}{[u(1+\alpha v)+v]^2}-d 
\end{pmatrix}
\]
Now, in order to analyze the stability for $(1,0)$, we substitute $u=1$ and $v=0$ in our general form of the Jacobian for boundary solution	. As a result, our Jacobian now simplifies to the following matrix:
\[J=
\begin{pmatrix}
a(b-1) &   -1\\
0 &   c-d
\end{pmatrix}
\]
Since we know that our sufficient and necessary criteria for stability, where $trace(J)<0$ and $det(J)>0$, we impose those conditions in our $J$. This gives us the following inequality:
\begin {equation}
\begin{split}
a(b-1)+c-d<0 \\
a(b-1)(c-d)>0
\end{split}
\end{equation}
From, algebraic analysis, it is evident that $(1,0)$ is always stable if $c<d$.
\\
Again, we compute the Jacobian for $(u,v)=(b,0)$, which simplifies to the following matrix:
\[J=
\begin{pmatrix}
ab(1-b) &   -1\\
0 &   c-d
\end{pmatrix}
\]
Similarly, we impose our criteria for stability, where $trace(J)<0$ and $det(J)>0$. We arrive at the following inequalities:
\begin {equation}
\begin{split}
ab(1-b)+c-d < 0 \\
ab(1-b)(c-d) > 0
\end{split}
\end{equation}
Again, through algebraic analysis of the inequality above, we conclude that $(b,0)$ is always unstable for $b>0$ and $c<d$.

\end{proof}

\subsection{Stability for Interior Solution}
We re-iterate our model but in a different form, which simplifies our analysis.
\begin {equation}
\begin{split}
\frac{du}{dt} &= \Big[\big(u-b\big)\big(1-u\big)-\frac{v\big(1+\alpha v\big)}{v+\big(1+\alpha v\big)}\Big]u = \Big[g_1\Big] u \\
\frac{dv}{dt} &= \Big[\frac{cu\big(1+\alpha v \big)}{v+\big(1+\alpha v \big)u}-d\Big]v = \Big[g_2 \Big] v
\end{split}
\end{equation}
such that,
\begin{equation}
g_1 = \Big[\big(u-b\big)\big(1-u\big)-\frac{v\big(1+\alpha v\big)}{v+\big(1+\alpha v\big)}\Big],  \hspace{1cm}
g_2 = \Big[\frac{cu\big(1+\alpha v \big)}{v+\big(1+\alpha v \big)u}-d\Big]
\end{equation}

\begin{theorem}
An interior solution is stable if $\frac{\partial g_1}{\partial u}<0$ for $c>0$
\end{theorem}

We now define our Jacobian for interior solution as follows:
\[J=
\begin{pmatrix} 
u\frac{\partial g_{1}}{\partial u} & u \frac{\partial g_{1}}{\partial v} \\  
\\
v\frac{\partial g_{2}}{\partial u} & v\frac{\partial g_{2}}{\partial v} 
\end{pmatrix}
\]
\begin{proof}
Assume,  $\frac{\partial g_1}{\partial u}<0$. Then, $J_{11}$ and $J_{22}$ is negative for positive $c$. This implies, the trace $J_{11} + J_{22} < 0 $. Similarly, $J_{12}$ is positive and $J_{21}$ is negative, and their product $J_{12}J_{21}$ is negative. This implies, the determinant, $J_{11}J_{22}-J_{12}J_{21}$ is positive. Both conditions are satisfied for stability. Hence, $\frac{\partial g_1}{\partial u}<0$ is a sufficient condition for stability of interior equilibrium.
\end{proof}

%
%
\section{Numerical Results}
In this section, we present our results from numerical experiments performed in MATLAB. 
\begin{figure}[H]%
\centering
\subfloat[\scriptsize{Stable Node, $\alpha =0.92$}]{{\includegraphics[width=.27\textwidth]{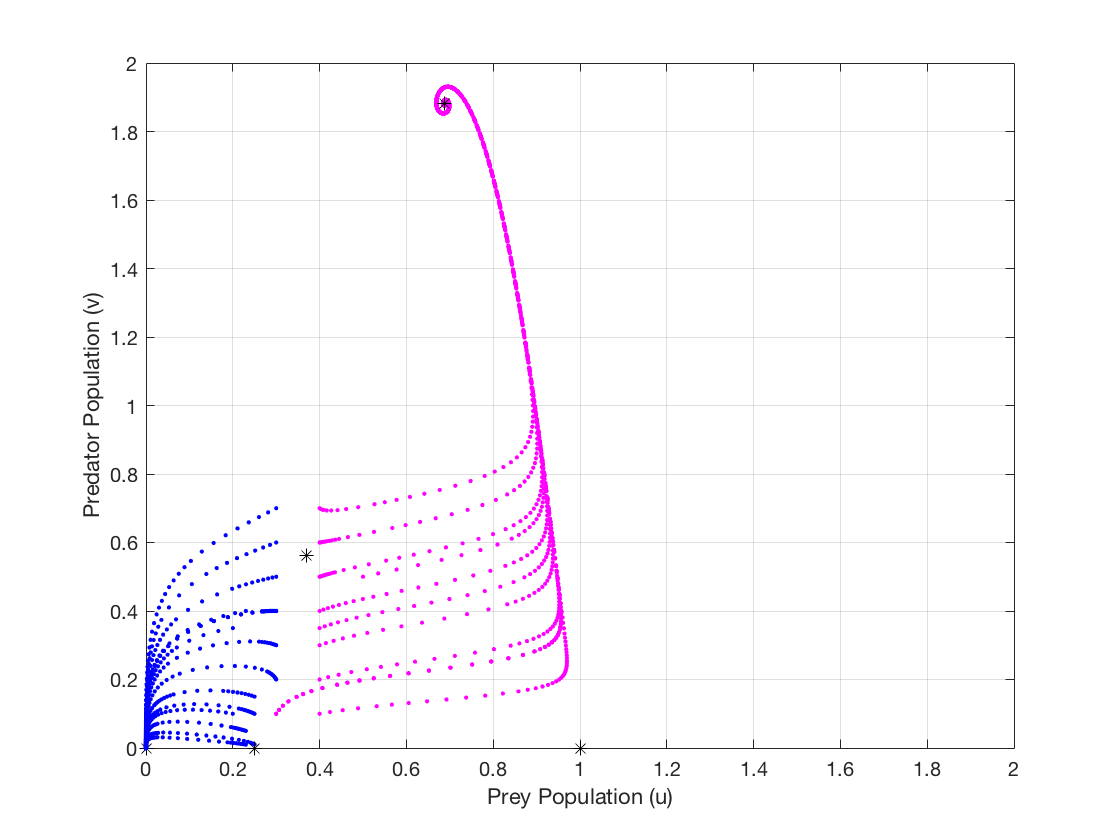} }}%
\qquad
\subfloat[\scriptsize{Unstable Node, $\alpha =0.96$}]{{\includegraphics[width=.27\textwidth]{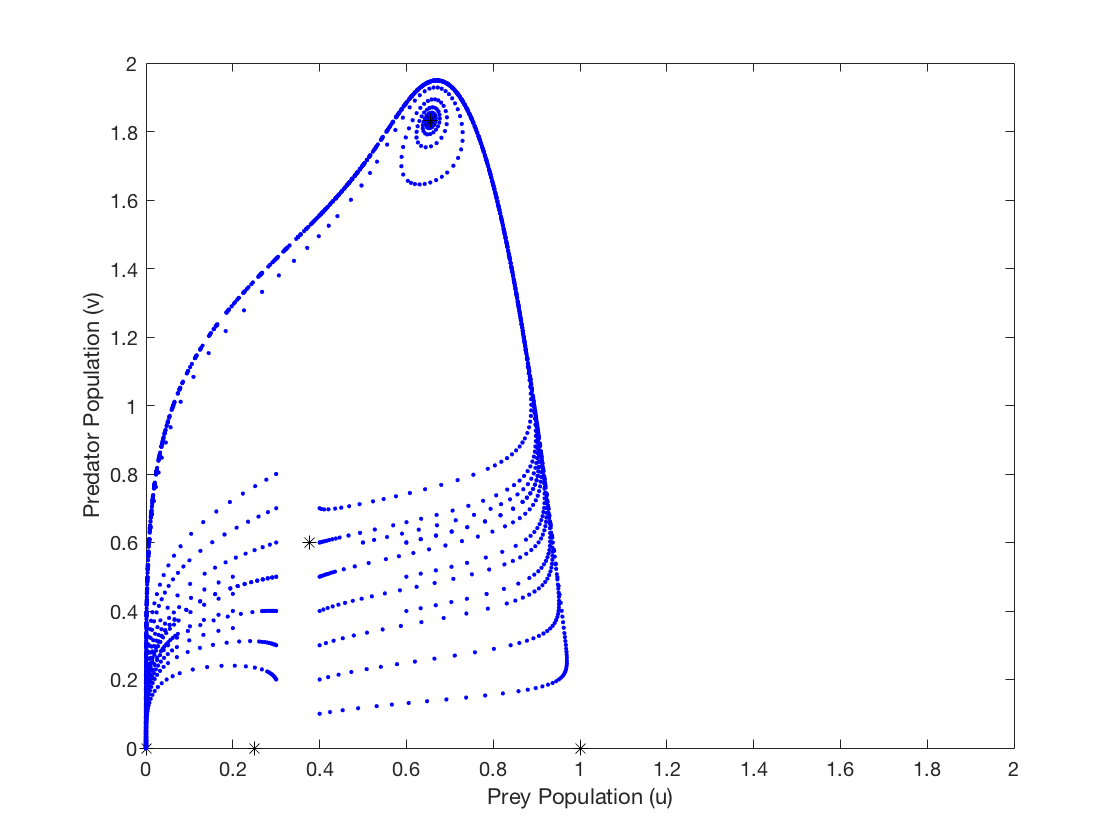} }}%
\caption{\footnotesize{Numerical simulation showing change in stability of equilibrium coexistence from $\alpha =0.92$ to  $\alpha =0.96$, given $a=10, b=0.25, c=2, d=1$.}}
\label{fig:ex}%
\end{figure}

\begin{figure}[H]%
\centering
\subfloat[\scriptsize{$b=-0.25$}]{{\includegraphics[width=.18\textwidth]{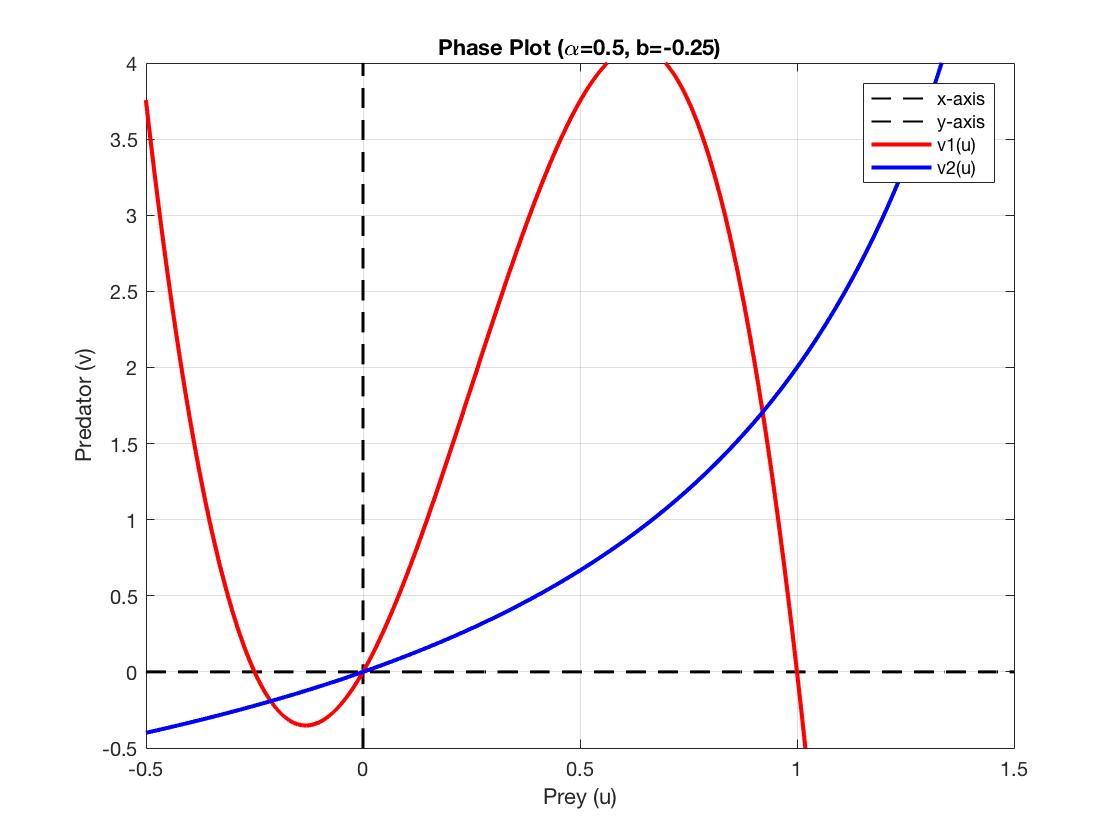} }}%
\subfloat[\scriptsize{$b=0$}]{{\includegraphics[width=.18\textwidth]{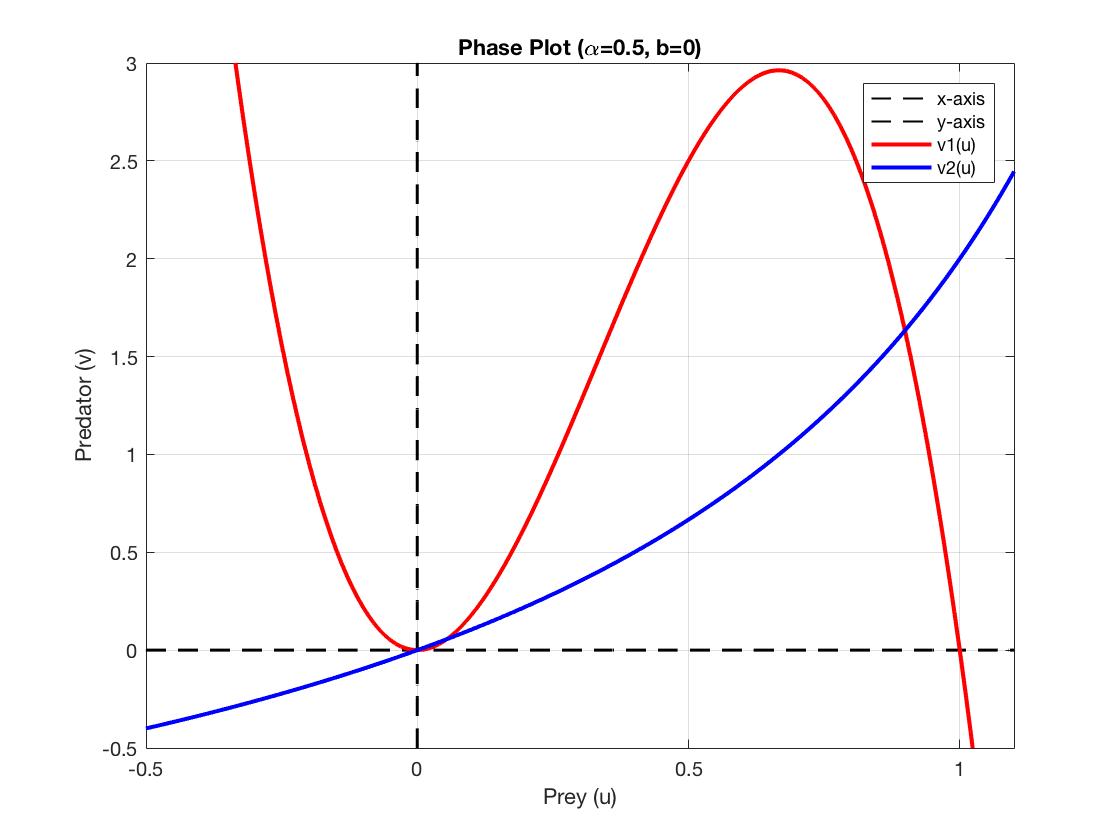} }}%
\subfloat[\scriptsize{$b=0.25$}]{{\includegraphics[width=.18\textwidth]{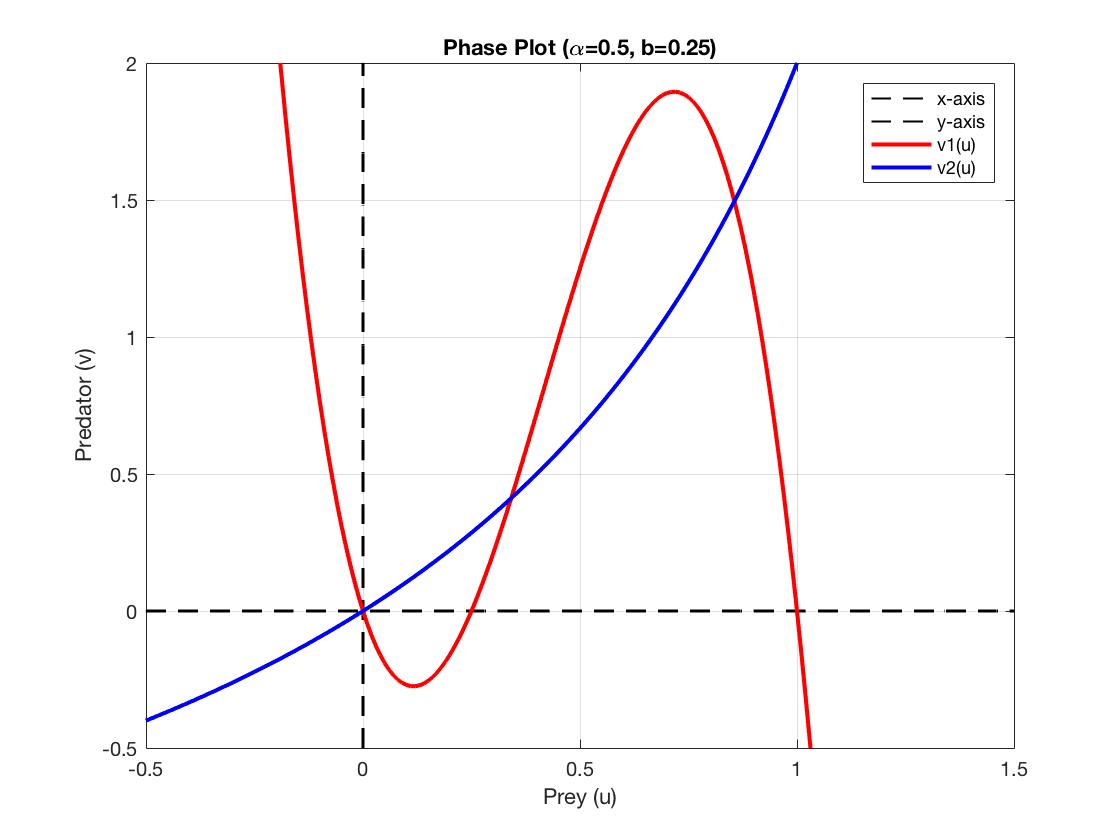} }}%
\subfloat[\scriptsize{$b=0.75$}]{{\includegraphics[width=.18\textwidth]{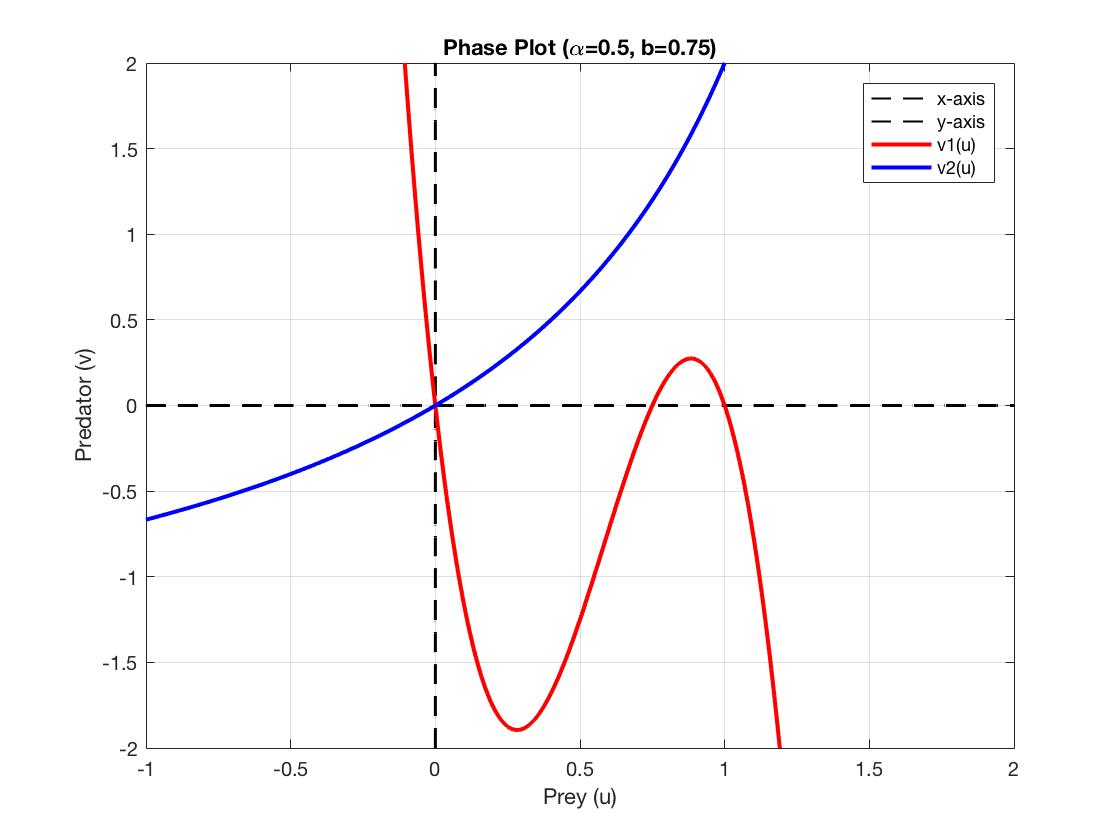} }}%
\caption{\footnotesize{Number of interior equilibrium solutions could be none, one and two. Case I: $\alpha=0.5$. (a) Intersection:$(0.9209, 0.9209)$; (b) Intersection:$(0.0544, 0.0544)$, $(0.8990, 1.6330)$; (c) Intersection:$(0.3146, 0.4119)$, $(0.8557, 1.4957)$; (d) Intersection: None; }}%
\label{}%
\end{figure}

\begin{figure}[H]%
\centering
\subfloat[\scriptsize{$b=-0.25$}]{{\includegraphics[width=.20\textwidth]{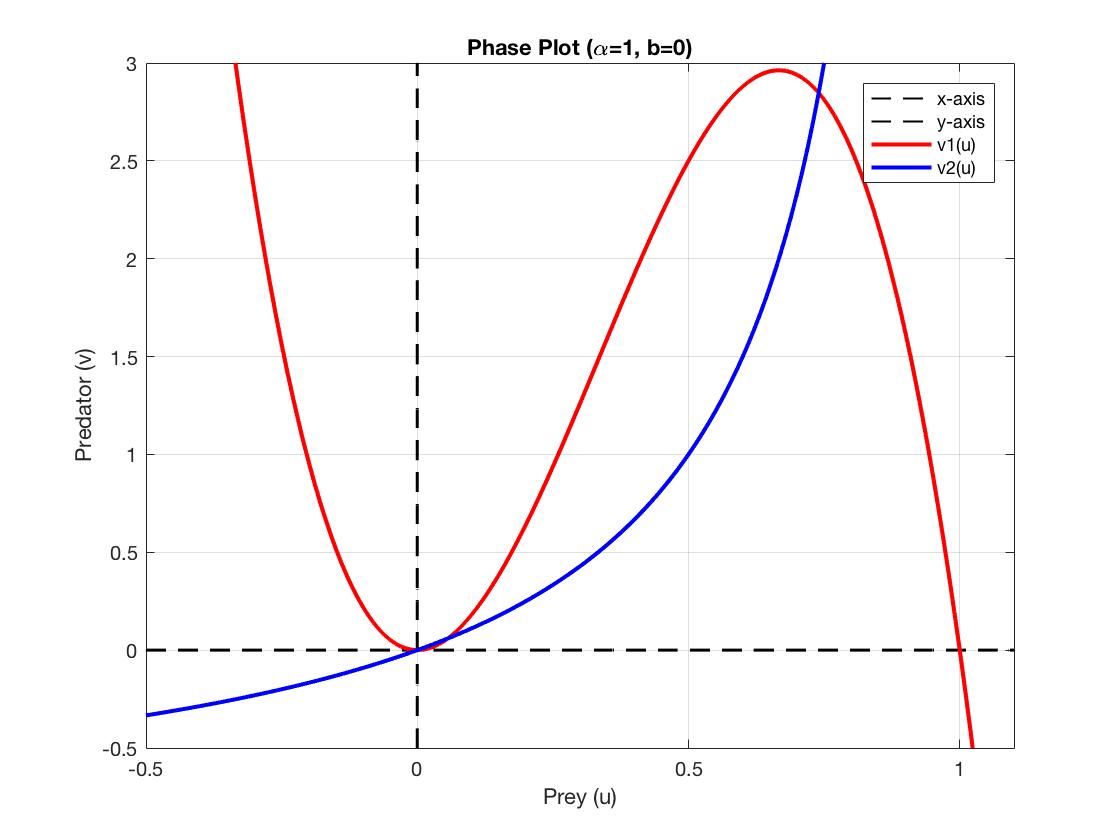} }}%
\subfloat[\scriptsize{$b=0$}]{{\includegraphics[width=.20\textwidth]{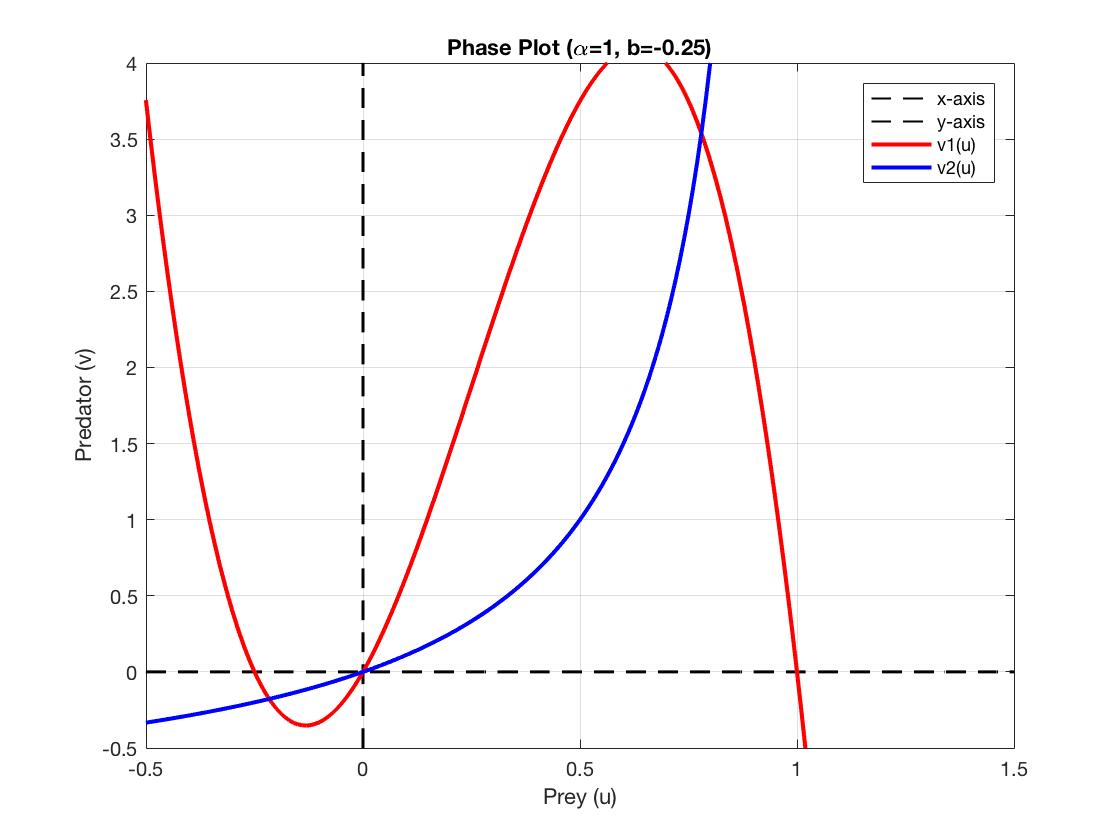} }}%
\subfloat[\scriptsize{$b=0.23$}]{{\includegraphics[width=.20\textwidth]{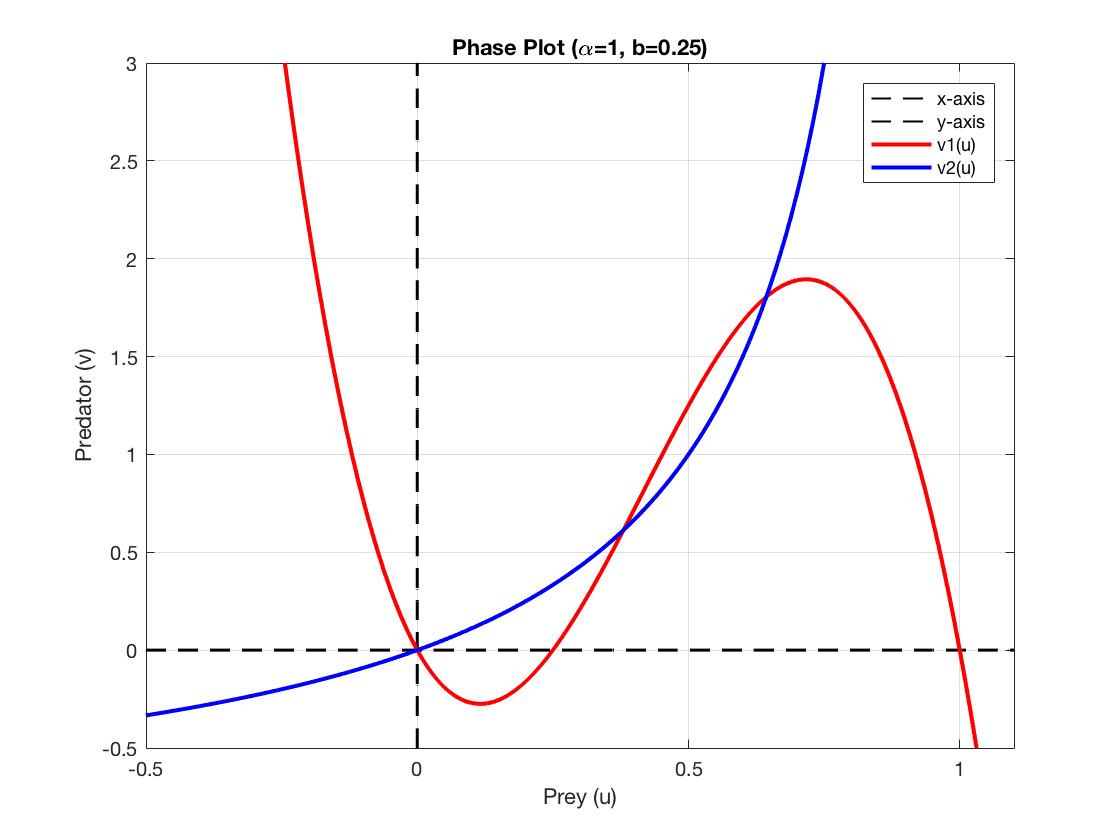} }}%
\subfloat[\scriptsize{$b=0.75$}]{{\includegraphics[width=.20\textwidth]{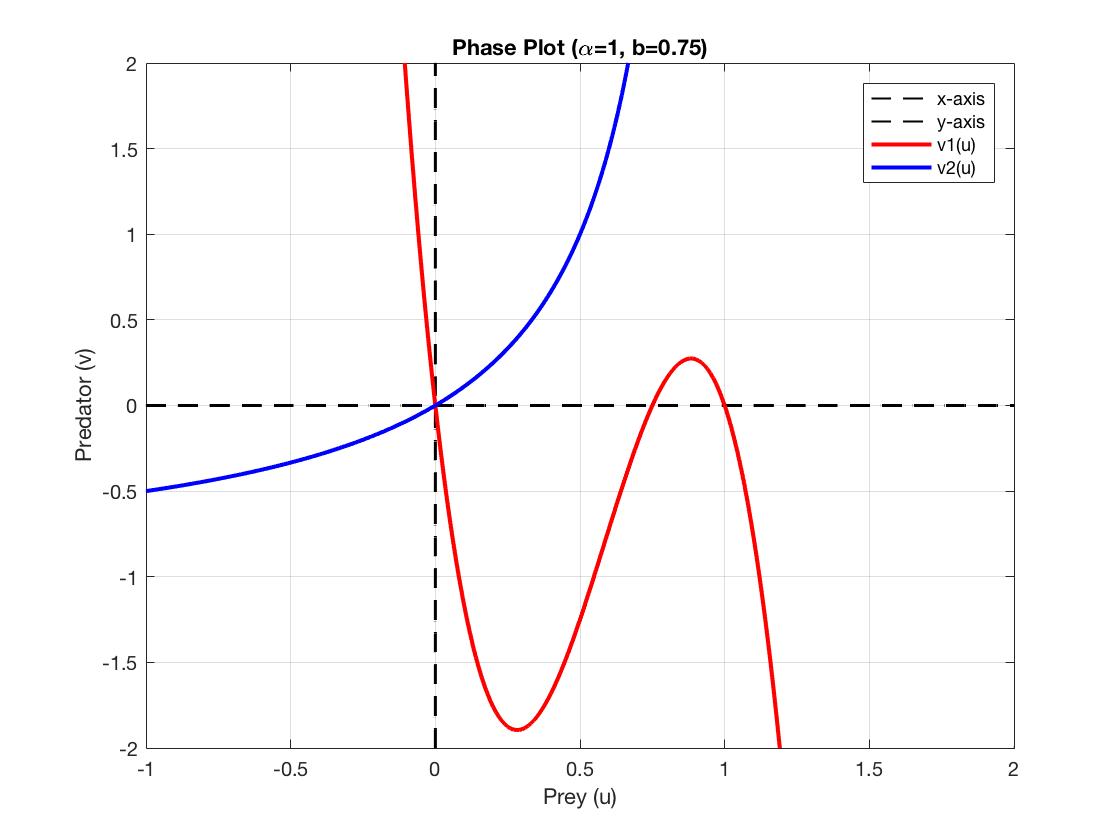} }}%
\caption{\footnotesize{
Case II: $\alpha=1$. (a) Intersection: $(0.7796, 3.5379)$; (b) Intersection:$(0.0595, 0.0595)$, $(0.7401, 2.8473)$; (c) Intersection:$(0.3801, 0.0.6132)$, $(0.6436, 1.8056)$; (d) Intersection: None;}}%
\label{fig:example}%
\end{figure}

\section{Conclusion}
In our research, two types of bifurcation was noticed. The first type of bifurcation occurred where there was a change in the number of interior stable points as shown in Figure 2 and Figure 3. The second type bifurcation occurred when there was a change in the stability of an existing interior solution from stable to unstable as shown in Figure 1. Although, increasing the hunting cooperation may seem to increase the chance of catching a prey to increase survival of the predator, it can be deduced that higher value for hunting cooperation is actually detrimental leading to extinction of both species in the ecosystem.

\section{Future Research}
For our future research, we want to conduct a two parameter analysis that involves at least hunting cooperation with any other variable. With comprehensive numerical simulation, it may be suggestive if Allee Effect is actually induced by Hunting Cooperation.

\section{Acknowledgement}
I would like to extend my sincere thanks to my advisors Dr. Zhifu Xie and Dr. Huiqing Zhu for their constant mentoring during the seven-weeks REU 2018 research. I am very grateful to Wright W. and Annie Rea Cross Endowment Scholarship for financial support, MAA-NREUP and School of Mathematics and Natural Sciences at USM for research experience for undergraduates.


\begin{thebibliography}{2}
\bibitem{RK} 
Feng Rao, Yun Kang
\textit{The Complex Dynamics of a Diffusive Prey-Predator Model with an Allee Effect in Prey.}
Elsevier, Ecological Complexity, 2016, pp.123-144
\bibitem{AH} 
Mickael Alves, Frank Hilker. 
\textit{Hunting Cooperation and Allee Effects in Predators.} 
Elsevier, Journal of Theoretical Biology, 2017, pp.13-22
\end{thebibliography}
\end{document}